\documentclass[paper=a4,parskip=half,DIV11,onecolumn,notitlepage]{scrartcl}

\usepackage[english]{babel}																	
\usepackage[utf8]{inputenc}																
\usepackage[T1]{fontenc}																		
\usepackage{graphics}																
\usepackage{color}																					
\usepackage[babel,english=american]{csquotes}											
\usepackage{typearea}													
\usepackage{setspace}																				
\usepackage{amsmath, amstext}																
\usepackage[autooneside]{scrpage2}													
\usepackage{enumerate}							 
\usepackage{authblk}																				
\usepackage{listings}																				
\usepackage[																								
margin=5pt,labelfont=bf,parskip=5pt]{caption}								
\usepackage{colortbl}
\usepackage{booktabs}

\usepackage[charter]{mathdesign}														
\usepackage[scaled]{berasans}																
\usepackage[scaled]{beramono}																
\usepackage{microtype}

\addtokomafont{sectioning}{\rmfamily}

\usepackage{pgf,tikz}
\usetikzlibrary{arrows}
\usepackage{nndoubleletters}

\usepackage[hyperref,thmmarks,amsmath,thref]{ntheorem}

\theoremstyle{break}
\theoremnumbering{arabic}
\theorembodyfont{\normalfont}
\newtheorem{thm}{Theorem}
\newtheorem{lem}[thm]{Lemma}
\newtheorem{cor}[thm]{Corollary}
\newtheorem{prop}[thm]{Proposition}

\theorembodyfont{\normalfont}

\theoremstyle{nonumberplain}
\theoremsymbol{\ensuremath{\square}}
\newtheorem{proof}{Proof.}
\newtheorem{proof2}{Proof}
\theoremsymbol{}

\usepackage{hyperref}

\allowdisplaybreaks[4]

\renewcommand{\epsilon}{\varepsilon}
\renewcommand{\phi}{\varphi}

\newcommand{\con}[1]{\hspace{0.2em}{\nwarrow\hspace{-1em}\nearrow}_{#1}}

\renewcommand{\phi}{\varphi}
\renewcommand{\epsilon}{\varepsilon}
\newcommand{\calC}{\mathcal{C}}
\newcommand{\calL}{\mathcal{L}}

\newcommand{\D}{\Delta}

\newcommand{\scalefactor}{0.65}

\hypersetup{
pdfborder={0 0 0},
pdfpagemode=UseOutlines,
bookmarksnumbered=true,
pdftitle=Longest Paths in Partial 2-Trees,
pdfauthor={Julia Ehrenmüller, Cristina G. Fernandes, Carl Georg Heise},
pdfproducer=Miktex 2.9,
}

\title{Nonempty Intersection of Longest Paths \\ in Series-Parallel Graphs}

\newcounter{f1}
\newcounter{f2}
\newcounter{f3}
\author[1,\fnsymbol{f1}]{Julia~Ehrenmüller}
\author[2,\fnsymbol{f3}]{Cristina~G.~Fernandes}
\author[1,\fnsymbol{f1},\fnsymbol{f2}]{Carl~Georg~Heise}
\affil[1]{\footnotesize Institut für Mathematik, Technische Universität Hamburg-Harburg, Germany, \{julia.ehrenmueller,carl.georg.heise\}@tuhh.de}
\affil[2]{\footnotesize Institute of Mathematics and Statistics, University of São Paulo, Brazil, cris@ime.usp.br}

\date{\today\\}

\begin{document}
\maketitle
\thispagestyle{empty}

\let\oldthefootnote\thefootnote
\renewcommand{\thefootnote}{\fnsymbol{footnote}}
\footnotetext[1]{The authors gratefully acknowledge the support of the Technische Universität München Graduate School's Thematic
Graduate Center TopMath.}
\footnotetext[2]{Partially supported by DFG grant GR 993/10-1.}
\footnotetext[3]{Partially supported by CNPq Proc.~308523/2012-1 and 477203/2012-4, FAPESP 2013/03447-6, and 
a joint CAPES-DAAD project (415/PPP-Probral/po/D08/11629,  Proj.~no.~333/09).}
\let\thefootnote\oldthefootnote

\begin{abstract}\noindent
In 1966 Gallai asked whether all longest paths in a connected graph have nonempty intersection.  This is not true in general and various
counterexamples have been found.  However, the answer to Gallai's question is positive for several well-known classes of graphs, as for instance
connected outerplanar graphs, connected split graphs, and 2-trees.  A graph is series-parallel if it does not contain $K_4$ as a minor.
Series-parallel graphs are also known as partial 2-trees, which are arbitrary subgraphs of 2-trees.  We present a proof that every connected
series-parallel graph has a vertex that is common to all of its longest paths.  Since 2-trees are maximal series-parallel graphs, and outerplanar
graphs are also series-parallel, our result captures these two classes in one proof and strengthens them to a larger class of graphs.  We also describe how this vertex can be found in linear time.
\end{abstract}

\section{Introduction}
\label{sec:intro}

A path in a graph is a \emph{longest path} if there exists no other path in the same graph that is strictly longer. The study of intersections of longest paths has a long history and, in particular, the question of whether every connected graph has a vertex that is common to all of its longest paths was raised by Gallai~\cite{Gallai68} in 1966. For some years it was not clear whether the answer is positive or negative until, finally, Walther~\cite{Walther69} found a graph on 25 vertices that answers Gallai's question negatively.

Today, the smallest known graph answering Gallai's question negatively is a graph on 12 vertices, found by Walther and
Voss~\cite{Walther74}, and independently by Zamfirescu~\cite{Zamfirescu76} (see Figure~\ref{fig:wzamf}). To see that the
depicted graph does not have a vertex common to all longest paths, one can identify the three leaves to obtain the
Petersen graph, which is hypohamiltonian, meaning that it does not have a Hamiltonian cycle but every vertex-deleted
subgraph is Hamiltonian.  Note that the length of a longest path in the depicted graph can be at most 10 since at most
two of its three leaves can be contained in a longest path.  But any path of length 10 in the depicted graph would
correspond to a Hamiltonian cycle in the Petersen graph.  It follows first that the length of a longest path is at
most~9 (and it is exactly~9) and second that the intersection of all longest paths is empty.

\begin{figure}[htbp]
	\centering
	\includegraphics[scale=\scalefactor,page=1]{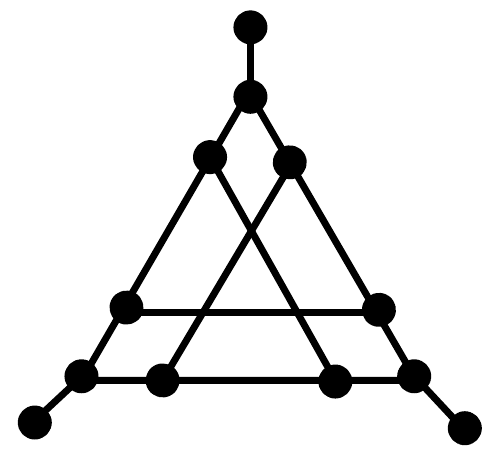}
	\caption{The counterexample of Walther, Voss, and Zamfirescu.}		
	\label{fig:wzamf}
\end{figure}

These are by far not the only counterexamples. In fact, there are infinitely many (even planar) ones since every hypotraceable graph, meaning a graph having no Hamiltonian path whose all vertex-deleted subgraphs have a Hamiltonian path, is obviously a counterexample. Thomassen proved in~\cite{Thomassen76} that there are infinitely many such graphs.

Since the answer to Gallai's question is negative in general, it seems natural to restrict the problem to subsets of a fixed size of all longest paths. It is well-known~\cite{ore} that any two longest paths of a connected graph share a common vertex. However, considering the intersection of more than two longest paths
gets more intriguing. It is still unknown whether any three longest paths of every connected graph share a common vertex. Zamfirescu asked this question several times~\cite{Voss91, Zamfirescu01} and it was mentioned at the 15th British Combinatorial Conference~\cite{BCC95}. It is presented as a conjecture in~\cite{HarrisHM2008} and as an open problem in the list collected by West~\cite{West}. Progress in this direction was made by de~Rezende, Martin, Wakabayashi, and the second author~\cite{deRezendeFMW13}, who proved that, if all non-trivial blocks of a connected graph are Hamiltonian, then any three longest paths of the graph share a vertex.
Skupie\'{n}~\cite{Skupien96} showed that for every \(p \geq 7\), there exists a connected graph such that~\(p\) longest paths have no common vertex and every~\(p-1\) longest paths have a common vertex. 

Even though it seems as if the property of having a vertex common to all longest paths is too strong, there are some classes of connected graphs for which this property holds. A simple example is the class of trees since in a tree all longest paths contain its center(s). Moreover, Klav{\v{z}}ar and Petkov{\v{s}}ek~\cite{KlavzarP90} proved that the intersection of all longest paths of a (connected) split graph is nonempty. Furthermore, they showed in~\cite{KlavzarP90} that, if every block of a connected graph \(G\) is Hamilton-connected, almost Hamilton-connected, or a cycle, then there exists a vertex common to all its longest paths. The latter result implies immediately that the answer to Gallai's question is positive for the class of (connected) cacti, where a graph is a cactus if and only if every block is either a simple cycle or a vertex or a single edge. 

In 2013, de Rezende et al.~proved the following two theorems.

\begin{thm}[\cite{deRezendeFMW13}]
\label{thm:outerplanar}
For every connected outerplanar graph, there exists a vertex common to all its longest paths.
\end{thm}

\begin{thm}[\cite{deRezendeFMW13}]
\label{thm:2-trees}
For every 2-tree, there exists a vertex common to all its longest paths.
\end{thm}

Theorem~\ref{thm:outerplanar} is a strengthening of a theorem by Axenovich~\cite{Axenovich09}, which states that any three longest paths in a connected outerplanar graph share a vertex. 

In this paper we treat the general case of nonempty intersection of all longest paths and prove that the answer to Gallai's question is positive for the class of connected series-parallel graphs settling a question raised in~\cite{deRezendeFMW13}. Note that a joint paper presenting two different proofs for this statement is to appear in \cite{us}. Since not only all trees and cacti but also outerplanar graphs and 2-trees are series-parallel, our result gives a unified proof for Theorems~\ref{thm:outerplanar} and~\ref{thm:2-trees} and generalizes them to a larger class of graphs. 

The rest of the paper is organized as follows. In Section~\ref{sec:definitions} we give essential definitions and prove some statements that will be useful in what follows. In Section~\ref{sec:series-parallel} we prove the main theorem by proceeding in three steps. First we fix a 2-tree that has the given series-parallel graph as a spanning subgraph. The 2-tree captures the structure of the given graph and guides us in the proof. The main obstacle is that we are not sure of which edges of this 2-tree truly exist in the given graph, so the techniques used for 2-trees in~\cite{deRezendeFMW13} fail. Therefore, we are somehow obliged to work with edges that exist in the 2-tree and may or may not exist in the given series-parallel graph. Roughly, the only things that we can rely on are the facts that our graph is connected, that two longest paths intersect, and on the special structure of the so-called components of the series-parallel graph, inherited from the 2-tree. In Section~\ref{sec:algo} we show that finding a vertex contained in all longest paths can be done in quadratic time for series-parallel graphs. Finally, in Section~\ref{sec:question} we state several open problems concerning the intersection of longest paths in specific classes of graphs.

\section{Preliminaries and definitions}
\label{sec:definitions}

We start with a few basic definitions which we use in the subsequent part of our paper. 
All graphs in this paper are undirected and finite. We write \(H\subseteq G\) if the graph \(H\) is a subgraph of the graph \(G\). Also, we denote by \(V(G)\) the set of vertices of \(G\).

Let $G$ be a graph and $s$ and $t$ be two of its vertices. We say $G$ is \emph{series-parallel with terminals $s$ and~$t$} if it can be turned into~$K_2$ by a sequence of the following operations: replacement of a pair of parallel edges with a single edge that connects their common endpoints, or replacement of a pair of edges incident to a vertex of degree 2 other than~$s$ or~$t$ with a single edge. A graph~$G$ is \emph{2-terminal series-parallel} if there are vertices $s$ and~$t$ in~$G$ such that~$G$ is series-parallel with terminals~$s$ and~$t$. A graph $G$ is \emph{series-parallel} if each of its 2-connected components is a 2-terminal series-parallel graph. (See~\cite[Sec.~11.2]{BrandstadtLS99}.)

A \emph{2-tree} can be defined in the following way. A single edge is a 2-tree. If \(T\) is not a single edge, then \(T\) is a 2-tree if and only if there exists a vertex \(v\) of degree~2 such that its neighbors are adjacent and~\(T-v\) is also a 2-tree.  A graph is a \emph{partial 2-tree} if it is a subgraph of a 2-tree. (See~\cite[Sec.~11.1]{BrandstadtLS99}.)  We say a partial 2-tree is \emph{trivial} if it consists of a single vertex or a single edge. Note that every edge in a non-trivial 2-tree is contained in a triangle.

It is well-known that a graph is a partial 2-tree if and only if it is \(K_4\)-minor free. 
Partial 2-trees are exactly the series-parallel graphs,
and are also known for being the graphs with tree width at most~2. (See~\cite[Sec.~11.1]{BrandstadtLS99}.) 

Next we present some notation we use in our proofs.

The \emph{length} of a path \(P\), denoted by \(|P|\), is the number of edges in \(P\). Let \(L(G)\) denote the length of a longest path in the graph \(G\). Let \(\mathcal L(G)\) denote the set of all longest paths in \(G\), that is, \(\mathcal L(G) = \{P \ | \) $P$ is a path in $G$ and $|P| = L(G)\}$. If the graph $G$ is clear from the context, we simply write $L$ for \(L(G)\) and $\mathcal L$ for $\mathcal L(G)$. 

By the intersection \(P \cap P'\) of two paths \(P\) and \(P'\), we mean the intersection of their vertex sets. If~\(v\) is a vertex of the path \(P\), we write \(v \in P\). If \(P\) and \(P'\) have a common endpoint \(x\) but no other common vertex, then the \emph{union} \(P \cup P'\) is simply defined as the path obtained by concatenating the path \(P\) and the path \(P'\) at the vertex~\(x\).

A subpath of a path $P$ is called a \emph{tail} of $P$ if it contains an endpoint of $P$. Given a vertex \(x\) in \(P\), the path \(P\) can be split into two subpaths~\(P'\) and~\(P''\) such that \(P' \cap P'' = \{x\}\); we call them \emph{tails of \(P\) starting at \(x\)}. If \(|P'| \geq |P''|\), then~\(P'\) is called a \emph{longer tail} of \(P\) starting at \(x\).

Given a second path \(Q\) such that \(Q \cap P \neq \varnothing\) and such that at least one endpoint \(x\) of \(P\) is not contained in \(Q\), we define the \emph{bridge path} \(P\con{x} Q\) as the path starting at the endpoint \(x\) going along \(P\) until the first intersection with \(Q\). 
 
For some subgraph \(H\), we define \(P[H]\) to be the induced subgraph of \(P\) in \(H\), that is, the collection of (maximal) subpaths of~\(P\) that lie in~\(H\). Note that this might be more than one path. 

In the next, we borrow some definitions and results presented by Tutte~\cite{Tutte01}. 
Let $G$ be a graph and $H$ be a subgraph of $G$. A \emph{vertex of attachment} of $H$ in $G$ is a vertex of $H$ that is incident to some edge of $G$ that is not an edge of $H$. Let \(T\) be a 2-tree and \(\{x,y\} \in E(T)\). 
An \emph{$\{x,y\}$-bridge} in $T$ is a minimal subgraph $B$ of $T$ containing a vertex other than $x$ and $y$ and whose vertices of attachment are contained in $\{x,y\}$. 
For each common neighbor \(z\) of \(x\) and \(y\), let~$B_{\{x,y\},z}(T)$ be the $\{x,y\}$-bridge in $T$ containing~$z$. 
There is a unique such bridge because, by Theorem~I.51~\cite{Tutte01}, the intersection of two distinct $\{x,y\}$-bridges in~$T$ lies in $\{x,y\}$. The \emph{interior} of the $B_{\{x,y\},z}(T)$ is the graph $B^\circ_{\{x,y\},z}(T) = B_{\{x,y\},z}(T) - x - y$. 

In what follows, for a series-parallel graph \(G=(V,E)\),  we let \(T(G)=(V,F)\) denote an arbitrary but fixed 2-tree that contains \(G\) as a spanning subgraph. We say a \emph{virtual edge/triangle} of \(G\) is an edge/triangle in \(T(G)\) independent of its existence in~\(G\). 

We denote by \(C_{\{x,y\},z}(G)\) the maximal subgraph of~\(G\) contained in \(B_{\{x,y\},z}(T(G))\). 
Note that such a subgraph of \(G\) may be disconnected. We call this subgraph the \emph{component} of $G$ generated by the virtual edge $\{x,y\}$ in direction $z$. Similarly, \(C^\circ_{\{x,y\},z}(G) = C_{\{x,y\},z}(G) - x - y\) is called the \emph{interior} of the component \(C_{\{x,y\},z}(G)\). Again, if the graph is clear from the context, we write~\(C_{\{x,y\},z} = C_{\{x,y\},z}(G)\). 

Also, let \(\mathcal C_{\{x,y\}}(G) = \{C_{\{x,y\},z}(G): z\) is adjacent to $x$ and $y$ in $T(G)\}$ be the set of all components generated by the virtual edge $\{x,y\}$. Further, define \(\mathcal C_{\{x,y\}|z}(G)=\mathcal C_{\{x,y\}}(G)\backslash\{C_{\{x,y\},z}(G)\}\) for a virtual triangle~\(\{x,y,z\}\).

A set of vertices \(W\) is called a \emph{Gallai set} (for \(G\)) if \(W \cap P\neq\varnothing\) for all \(P\in \mathcal L(G)\). If \(W\) is a Gallai set and if its vertices are pairwise connected by virtual edges, we call this set a \emph{virtual Gallai edge} and a \emph{virtual Gallai triangle}, when \(W\) has size two or three, respectively. 

We say a vertex \(v \in V\) is a \emph{Gallai vertex} if \(\{v\}\) is a Gallai set. Note that the intersection of all longest paths of a graph \(G\) is nonempty if and only if \(G\) has a Gallai vertex.

\newcommand{\Le}[2]{\calL_{{#1}{#2}}} 
\newcommand{\LeB}[2]{\calL_{{#1}\overline{#2}}}
\newcommand{\Luvw}[3]{\calL_{{#1}{#2}{#3}}}
\newcommand{\Lo}[3]{\calL_{({#1}{#2}{#3})}}
\newcommand{\LB}[3]{\calL_{{#1}{#2}\overline{#3}}}
\newcommand{\LBB}[3]{\calL_{{#1}\overline{#2}\overline{#3}}}

For a given virtual edge \(\{u,v\}\), let \(\Le{u}v = \{P \in \mathcal L \, | \, u, v \in P\}\) and \(\LeB{u}v = \{P \in \mathcal L \; | \; u \in P, \; v \notin P\}\). Similarly, for a given virtual triangle \(\D=\{u, v, w\}\), we define $\Luvw{u}vw = \{P \in \calL \; | \; u, v, w \in P\}$, $\LB{u}vw = \{P \in \calL \; | \; u, v \in P, \; w \not\in P\}$, and $\LBB{u}vw = \{P \in \calL \; | \; u \in P, \; v, w \not\in P\}$. Moreover, let~$\Lo{u}vw = \{P \in \Luvw{u}vw \; | \; v $ is between $u$ and $w$ in $P\}$.

We end this section with the following auxiliary results that will be useful in Section~\ref{sec:series-parallel}. Note that the first four results hold for general graphs, not only for series-parallel graphs.

\begin{prop}[\cite{ore}]\label{prop:ore}
Any two longest paths in a connected graph share a common vertex.
\end{prop}

\begin{lem}\label{lem:3paths2}
In a graph, let \(P_1\) and \(P_2\) be two paths with tails \(R_1\) and~\(R_2\), respectively (that is, subpaths containing an endpoint of \(P_1\) or
\(P_2\)) such that \(R_1\cap P_{2} =\varnothing\) and \(R_2\cap P_{1} =\varnothing\). If there exists a \emph{connecting path} \(P\) such that
\(\varnothing\neq P\cap P_1\subseteq R_1\) and \(\varnothing\neq P\cap P_2\subseteq R_2\), then \(P_1\) and \(P_2\) cannot both be longest paths.
\end{lem}
\begin{proof}
Assume for a contradiction that both \(P_1\) and \(P_2\) are longest paths. For \(i\in\{1,2\}\), let~\(R_i'\) denote the other tail of~\(P_i\) so that
\(P_i = R_i \cup R_i'\) and $R_i$ and $R'_i$ intersect at only one vertex.  By assumption, both~\(R_1\) and~\(R_2\) intersect \(P\).  Hence, there exist
vertices \(x\) and~\(y\) such that \(x \in R_1 \cap P\), \(y \in R_2 \cap P\), and the interior of the subpath of~\(P\) starting at~\(x\) and ending
in~\(y\) does not contain vertices in~\(P_1\) or~\(P_2\). Let~\(Q_1\) denote the path obtained from going along \(R_1'\), along \(R_1\) until~\(x\),
along \(P\) until \(y\), and then along \(R_2\) until the endpoint that is not in \(R_2'\). Let \(Q_2\) denote the path obtained from going along
\(R_2'\), along \(R_2\) until \(y\), along \(P\) until~\(x\), and then along~\(R_1\) until the endpoint that is not in \(R_1'\).  Now, as $|Q_1|+|Q_2|
> |P_1|+|P_2|$, we have that \(|Q_1|>|P_1|\) or~\(|Q_2|>|P_2|\), a contradiction.
\end{proof}

\begin{figure}[htbp]
	\centering
		\includegraphics[scale=\scalefactor]{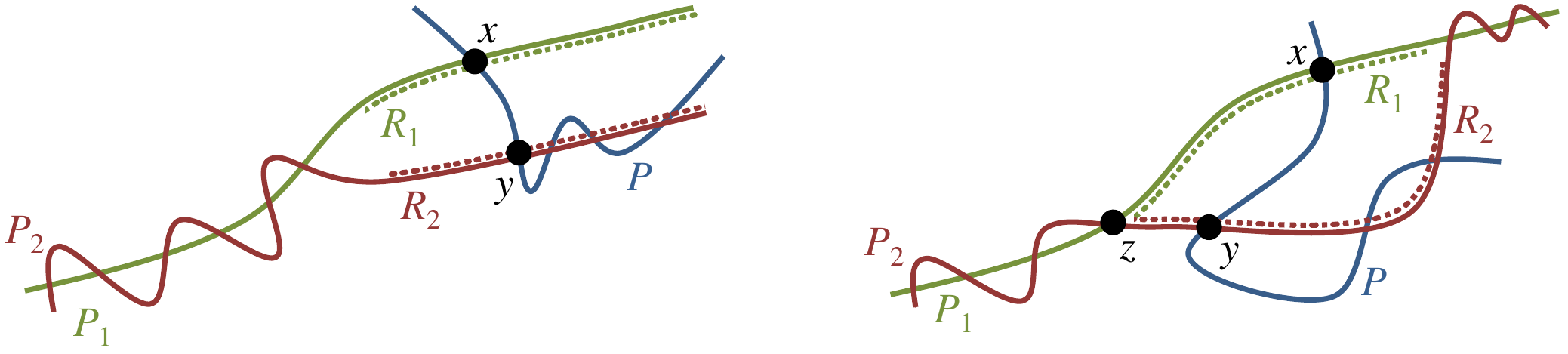}
	\caption{Situation for Lemmas~\ref{lem:3paths2} and~\ref{lem:3paths}.}
	\label{fig:lem34}
\end{figure}

\begin{lem}\label{lem:3paths}
In a graph, let \(P_1\) and \(P_2\) be two paths that share a common vertex \(z\) and let~\(R_1\) and~\(R_2\) be two subpaths of~\(P_1\) and~\(P_2\),
respectively, both having $z$ as an endpoint, such that \(R_1\cap P_2 = \{z\}\) and \(R_2 \cap P_1 = \{z\}\). If there exists a \emph{connecting path}
\(P\) such that \(z\notin P\), \(\varnothing\neq P\cap P_1\subseteq R_1\), and~\(\varnothing\neq P\cap P_2\subseteq R_2\), then \(P_1\) and \(P_2\)
cannot both be longest paths.
\end{lem}
\begin{proof}
Assume for a contradiction that both \(P_1\) and \(P_2\) are longest paths.  By assumption, both~\(R_1\) and \(R_2\) intersect \(P\) in a vertex other
than \(z\).  Hence, there exist vertices \(x\) and \(y\) distinct from \(z\) such that \(x \in R_1 \cap P\), \(y \in R_2 \cap P\), and the interior of
the subpath of \(P\) starting at~\(x\) and ending in~\(y\) does not contain vertices in~\(P_1\) or~\(P_2\). Let \(\tilde R_1\) denote the path
starting at~\(z\), going along \(R_1\), and ending in \(x\), and let \(\tilde R_2\) denote the path starting at \(z\), going along \(R_2\), and ending
in \(y\). Let \(R_1'\) and \(R_2'\) denote the tails of~\(P_1\) and~\(P_2\) starting at~\(z\) not containing~\(R_1\) and~\(R_2\), respectively.  If
\(|\tilde R_1| \geq |\tilde R_2|\), then by combining \(R_2'\), \(\tilde R_1\), the subpath of~\(P\) starting at~\(x\) and ending in~\(y\), and the
tail of~\(P_2\) starting at~\(y\) and not containing~\(z\), we get a path strictly longer than~\(P_2\), a contradiction.  If, on the other hand,
\(|\tilde R_1| < |\tilde R_2|\), then by combining \(R_1'\), \(\tilde R_2\), the subpath of~\(P\) starting at~\(y\) and ending in~\(x\), and the tail
of \(P_1\) starting at~\(x\) and not containing \(z\), we get a path strictly longer than~\(P_1\), a contradiction.
\end{proof}

Observe that Lemma~\ref{lem:3paths2} is not a consequence of Lemma~\ref{lem:3paths}. Indeed, in the situation of Lemma~\ref{lem:3paths2}, starting
from $x$ and going along $P_1$, the first common vertex with $P_2$ might not be the same as the first common vertex of $P_2$ with $P_1$, starting
from $y$. Thus, the vertex $z$ as required in Lemma~\ref{lem:3paths} might not exist.

At some points of the proofs in the next section, we are in a situation where one of the two lemmas above apply. 
The next corollary describes this situation.

\begin{figure}[htbp]
	\centering
		\includegraphics[scale=\scalefactor]{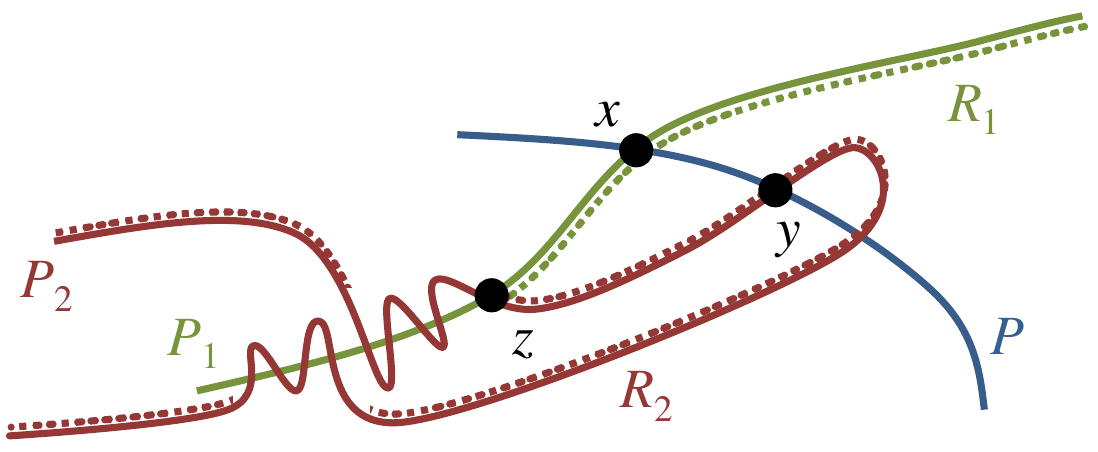}
	\caption{Situation for Corollary~\ref{cor:3paths3}.}
	\label{fig:cor5}
\end{figure}

\begin{cor}\label{cor:3paths3}
In a graph, let \(P_1\) and \(P_2\) be two paths that share a common vertex~\(z\) and let~\(R_1\) be a tail of~\(P_1\) starting at~\(z\). Let~\(R_2\)
be a union of pairwise internally vertex disjoint subpaths of~\(P_2\) (that is, they may have common endpoints) such that all paths in~\(R_2\) have as
one endpoint~\(z\) or an endpoint of~\(P_2\).  Suppose \(R_1\cap P_2 = \{z\}\) and \(R_2 \cap P_1 \subseteq \{z\}\).  If there exists a
\emph{connecting path} \(P\) such that \(z\notin P\), \(\varnothing\neq P\cap P_1\subseteq R_1\), and \(\varnothing\neq P\cap P_2\subseteq R_2\),
then \(P_1\) and \(P_2\) cannot both be longest paths.
\end{cor}

\begin{proof}
There exist vertices \(x \in P \cap R_1\) and \(y \in P \cap R_2\) such that the interior of the subpath~\(P^{x,y}\) of $P$ starting at~\(x\) and
ending in \(y\) does not contain any other vertices in~\(P_1\) or~\(P_2\). Let~\(R_2'\) be the path in~\(R_2\) that contains \(y\). If \(R_2'\)
contains~\(z\) then the statement follows from Lemma~\ref{lem:3paths} for longest paths~\(P_1\) and~\(P_2\) with their subpaths~\(R_1\) and~\(R_2'\),
respectively, and connecting path \(P^{x,y}\). Otherwise, the statement follows from Lemma~\ref{lem:3paths2} again for longest paths~\(P_1\)
and~\(P_2\), tails~\(R_1\) and~\(R_2'\), and connecting path~\(P^{x,y}\).
\end{proof}

The next two results are specific for series-parallel graphs. 

\begin{lem}\label{lem:SP}
Let $\D=\{v_1,v_2,v_3\}$ be a virtual triangle in a connected series-parallel graph $G$. 
If $R_i$ is a path in $G$ with $v_i$ as an endpoint and $R_i \cap \D = \{v_i\}$ for each $i\in\{1,2,3\}$,
then only one of the sets $R_1\cap R_2$, $R_1\cap R_3$, and $R_2\cap R_3$ can be nonempty. 
Furthermore, if $R_i\cap R_j\neq\varnothing$, then $R_i\cup R_j\subseteq C$ for some component $C\in\calC_{\{v_i,v_j\}|v_k}$.
\end{lem}
\begin{proof}
For the first statement, assume without loss of generality that $R_1\cap R_2\neq\varnothing$. 
Then we have a path $S$ from $v_1$ to $v_2$ consisting of $R_1 \con{v_1} R_2$ and the tail of $R_2$ containing $v_2$.
Note that this path does not use $v_3$ or the virtual edge $\{v_1,v_2\}$ because $R_i$ contains only $v_i$ in $\D$ for $i=1,2$.  

If additionally $R_1\cap R_3\neq\varnothing$ or $R_2\cap R_3\neq\varnothing$, then $v_3$ is connected to $S$ by a path $S'\cup S''$, where~$S'$ is the
shortest tail of $R_3$ from $v_3$ to a vertex $u$ in $R_1\cup R_2$, and $S''$ is a shortest path from $u$ to $S$ in the connected graph $R_1\cup R_2$.
Let $x$ be the endpoint of $S''$ in $S$. Observe that $x$ is an internal vertex of $S$. So~$\{x,v_1,v_2,v_3\}$ determines a $K_4$ minor in $T(G)$, a
contradiction.

For the second statement, suppose that $R_i$ and $R_j$ intersect. Obviously, $H=(R_i\cup R_j)-v_i-v_j$ must lie in the interior of a
$\{v_i,v_j\}$-bridge $B$ of~$T(G)$. Also, the edge $\{v_i,v_j\}$ is a cut set in~$T(G)$, separating~$v_k$ from~$H$. Otherwise we would have three
paths as above, namely $R_i$, $R_j$, and the path from~$v_k$ to~$H$ avoiding $v_i$ and $v_j$, and at least two of the pairs within these three paths
would intersect.  Therefore $B \in \calC_{\{v_i,v_j\}|v_k}(T(G))$ and thus $R_i\cup R_j \subseteq C=G[V(B)] \in \calC_{\{v_i,v_j\}|v_k}(G)$.
\end{proof}

\begin{lem}\label{lem:SP2}
Let $\D=\{v_1,v_2,v_3\}$ be a virtual triangle in a connected series-parallel graph $G$, and~$R$ be a path in $G$ with $v_i$ as endpoint and $R \cap
\D = \{v_i\}$, for some $i$ in $\{1,2,3\}$. Let $j$ and $k$ be such that $\{i,j,k\}=\{1,2,3\}$.  If $S_1$ is a path with endpoints $v_i$ and $v_j$
such that $S_1 \cap \D = \{v_i,v_j\}$, and $S_2$ is a path with endpoints $v_j$ and $v_k$ such that $S_2 \cap \D = \{v_j,v_k\}$, then $R \cap S_2 =
\varnothing$ and $S_1\cap S_2=\{v_j\}$.
\end{lem}
\begin{proof}
Assume for a contradiction that there is some vertex $x \in R \cap S_2$. Split the path $S_2$ at $x$ and look at the two tails $S_2^j$ and $S_2^k$
starting at $x$ and ending at $v_j$ and $v_k$, respectively. Now $R$, $S_2^j$, and $S_2^k$ are three paths as in Lemma~\ref{lem:SP} and
they all intersect at $x$, a contradiction.

Similarly, if $y\in (S_1\cap S_2)\backslash\{v_j\}$, split $S_2$ analogously at $y$ obtaining $S_2^j$ and $S_2^k$. Then $S_1-v_j$, $S_2^j$,
and~$S_2^k$ are three paths as in Lemma~\ref{lem:SP} and they all intersect at $y$, again a contradiction.
\end{proof}

\section{Intersection of longest paths in series-parallel graphs}
\label{sec:series-parallel}


As we have already mentioned in Section~\ref{sec:intro}, de~Rezende at el.~\cite{deRezendeFMW13} proved that the intersection of all longest paths of a 2-tree is nonempty. In this section, we extend this result proving that all connected subgraphs of 2-trees, that is, all series-parallel graphs, have also this property. We proceed in four steps. First, we prove in Lemma~\ref{lem:triang} that there exists a virtual Gallai triangle. Then, we show in Lemma~\ref{lem:edge} that actually one virtual edge of this triangle is a virtual Gallai edge and there exists a component generated by this virtual edge that satisfies certain properties. In Lemma~\ref{lem:iter2} we prove that either one of the endpoints of this virtual edge is a Gallai vertex or we can find an adjacent virtual Gallai edge and a strictly smaller component satisfying the same properties. By iterating, we end up with a Gallai vertex since we only consider finite graphs. 
%

\begin{lem}\label{lem:triang}
In every non-trivial connected series-parallel graph \(G\), 
there exists a virtual Gallai triangle.
\end{lem}

\begin{proof}
Take any virtual triangle \(\D_{0}\) of a non-trivial connected series-parallel graph \(G\). Note that, in every connected series-parallel graph, every virtual edge is either a cut set of $G$ or is contained in exactly one virtual triangle. Assume that there exists a longest path \(P_0\) in \(G\) containing no vertex of \(\D_{0}\). Then there exists a virtual edge $e_0 \subseteq \D_0$, which is a cut set, and a vertex $z_0 \notin \D_0$ such that $z_0$ is adjacent to both endpoints of $e_0$ in $T(G)$ and $P_0$ lies in the component generated by $e_0$ in direction $z_0$, that is, $P_0 \subseteq C_{e_0,z_0}^\circ$. By Proposition~\ref{prop:ore}, all longest paths must intersect \(P_0\) and so they have at least one vertex in $C_{e_0,z_0}^\circ$. Note that \(\D_1 = e_0 \cup \{z_0\}\) is a triangle in $T(G)$ and thus a virtual triangle in \(C_{e_0,z_0}\). 
Now either all longest paths contain a vertex of \(\D_1\) and we are done, or there exist a longest path $P_1$, a virtual edge $e_1 \subseteq \D_1$, where $e_1 \neq e_0$ and $e_1$ is a cut set, and a vertex $z_1 \notin \D_1$ such that $z_1$ is adjacent to both endpoints of~$e_1$ in $T(G)$ and \(P_1 \subseteq C_{e_1,z_1}\). Note that \(C_{e_1,z_1} \subsetneq C_{e_0,z_0}\), as~\(P_1\) must intersect \(P_0\) in \(C_{e_0,z_0}^\circ\) again by Proposition~\ref{prop:ore}. Iteratively, obtain \(\D_2\) and \(C_{e_2,z_2}\) and eventually a strictly decreasing sequence of components \(C_{e_0,z_0}\supsetneq C_{e_1,z_1} \supsetneq C_{e_2,z_2} \supsetneq \cdots \supsetneq C_{e_k,z_k}\). Since \(G\) is finite, this process ends with some triangle \(\D=\D_k\) such that all longest paths contain one vertex of \(\D\). 
\end{proof}

Next we prove that one of the virtual edges of a Gallai triangle is a virtual Gallai edge and there exists a component generated by this virtual edge that satisfies certain properties.


\begin{lem}\label{lem:edge}
For every connected series-parallel graph \(G=(V,E)\), there exists a Gallai vertex, or a virtual Gallai edge \(\{u, v\}\) and a component \(C \in
\mathcal C_{\{u,v\}}\) such that, for every pair \[(P, P') \in (\LeB{u}v \times \LeB{v}u) \cup (\LeB{u}v \times \Le{u}v) \cup (\LeB{v}u \times
\Le{u}v),\] there exists a vertex in \(C^\circ \cap P \cap P'\).
\end{lem}

Before presenting the proof of Lemma~\ref{lem:edge}, we prove an intermediate result stated in the next lemma.  Throughout the next proofs, we keep
Lemmas~\ref{lem:SP} and~\ref{lem:SP2} in mind and use them implicitly \emph{whenever} we claim that certain constructions are indeed paths and
whenever we claim that a path lies in a certain component.

For the proof of Lemmas~\ref{lem:triang3} and~\ref{lem:iter2}, we use the following notation. Every path \(P \in \LB{u}vw\) can be split at \(u\) and
\(v\), resulting in three subpaths. Let \(P^{(u)}\) and~\(P^{(v)}\) be the tails of~\(P\) starting at vertex \(u\) and vertex \(v\), respectively. The
remaining subpath, joining~\(u\) and~\(v\), is denoted by \(P^{(u,v)}\). Analogously, a path \(P\in \Lo{u}vw\) is split into \(P^{(u)}\),
\(P^{(u,v)}\), \(P^{(v,w)}\), and~\(P^{(w)}\) (see Figure~\ref{fig:split}).

\begin{figure}[!htbp]
 \centering
  \includegraphics[scale=0.55]{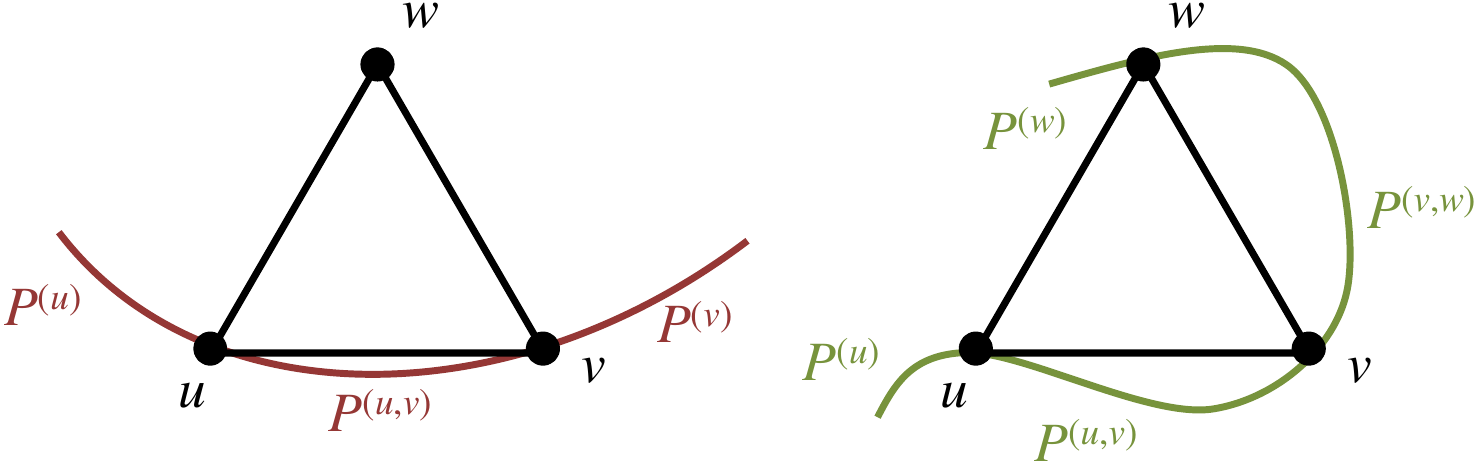}
  \caption{Splitting \(P \in \LB{u}vw\) at vertices \(u\) and \(v\), and \(P \in \Lo{u}vw\) at vertices \(u\), \(v\), and~\(w\).}
\label{fig:split}
\end{figure}


\begin{lem}\label{lem:triang3}
Let \(\D\) be a virtual Gallai triangle in a non-trivial connected series-parallel graph~\(G\). If \(\LB{x}{y}{z} \neq \varnothing\) 
for every $x$, $y$, $z$ such that $\{x,y,z\} = \D$, then, for some $u$, $v$, $w$ such that $\{u,v,w\} = \D$, there is a component \(C \in \mathcal
C_{\{u,w\}|v}\) such that, for every pair
 \[(P, P') \in \left( \bigcup_{\{x,y,z\}=\D} \LB{x}{y}{z} \times \LB{x}{z}{y} \right) 
                                            \cup (\LB{u}{v}{w} \times \Luvw{u}{v}{w}) \cup (\LB{v}{w}{u} \times \Luvw{u}{v}{w}),\] 
there exists a vertex in \(C^\circ \cap P \cap P'\).
\end{lem}

\begin{proof2}[Lemma~\ref{lem:triang3}]
Let $P \in \LB{u}{v}{w} \cup \LB{u}{w}{v} \cup \LB{v}{w}{u}$, where $\{u,v,w\} = \D$, and $x \in \D$ be such that~$x$ is in $P$ and $P^{(x)}$
is as long as possible. Without loss of generality, we may assume $P \in \LB{u}{v}{w}$ and $x=u$. 
In what follows, we use $P_z$ to refer to an arbitrary path in $\LB{x}{y}{z}$ where $\{x,y,z\} = \D$.
 
First note that~$P^{(u)}$ intersects every~$P_u^{(w)}$, otherwise $P^{(u)} \cup P^{(u,v)} \cup P_u^{(v,w)} \cup P_u^{(w)}$ is a path of length
strictly greater than~$L$ by the choice of~$P$.  Thus,~$P^{(u)}$ must lie in a component $C \in \calC_{\{u,w\}|v}$.  We will prove that $C$
has the property stated in the lemma.

We start by proving that each path in $\LB{u}{v}{w}$ intersects in $C^\circ$ every path in $\LB{v}{w}{u}$, that is, we show that each $P_w^{(u)}$ intersects
every $P_u^{(w)}$. Observe that $|P_w^{(u)}| = |P^{(u)}|$, otherwise $P^{(u)} \cup P_w^{(u,v)} \cup P_w^{(v)}$ would be a path of length strictly
greater than~$L$. So the argument previously applied to~$P$, now with~$P_w$ instead, implies that each~$P_w^{(u)}$ intersects every~$P_u^{(w)}$.

Now we prove that each path in $\LB{u}{v}{w}\cup \LB{v}{w}{u}$ intersects in $C^\circ$ every path in $\LB{u}{w}{v}\cup \Luvw{u}{v}{w}$.  
First note that each $P_w^{(u)}$ intersects $Q^{(u,w)}$ in~$C^\circ$ for every $Q \in \LB{u}{w}{v}\cup \Lo{u}{w}{v} \cup \Lo{v}{u}{w}$, otherwise 
$P_w^{(u)} \cup Q^{(u,w)} \cup P_u^{(w,v)} \cup P_u^{(v)}$ is a path of length strictly greater than~$L$ by the choice of~$P$. As both $\LB{u}{w}{v}$
and~$\LB{v}{w}{u}$ are nonempty, there exist at least one such $Q$ and one such $P_u$.
So $Q$ and~$P_u^{(w)}$ must intersect in~$C^\circ$, otherwise we derive a contradiction from Lemma~\ref{lem:3paths} for subpaths \(P_u^{(w)}\)
and~\(Q^{(u,w)}\), \(z=w\), and a connecting path contained in \(P_w^{(u)}\).  Second, we prove that each $P_w^{(u)}$ and $P_u^{(w)}$ intersect every
$Q \in \Lo{u}{v}{w}$ in $C^\circ$.  Observe that $|Q^{(u)}| \leq |P^{(u)}|$, otherwise either ${Q^{(u)} \cup P^{(u,v)} \cup P^{(v)}}$ or ${Q^{(u)} \cup
P_v^{(u,w)} \cup P_v^{(w)}}$ is a path of length strictly greater than~$L$ by the choice of~$P$.  If $|Q^{(u)}| < |P^{(u)}| = |P_w^{(u)}|$, then
$P_w^{(u)}$ intersects~$Q^{(w)}$, otherwise $P_w^{(u)} \cup Q^{(u,v)} \cup Q^{(v,w)} \cup Q^{(w)}$ is a path of length strictly greater
than~$L$. So~$P_u^{(w)}$ and $Q$ must intersect in~$C^\circ$, otherwise we derive a contradiction from Corollary~\ref{cor:3paths3} for longest paths
$P_u$ and $Q$ with \(z=w\), tail $P_u^{(w)}$, subpaths \(Q[C]\), and connecting path \(P_w^{(u)}\).  If $|Q^{(u)}| = |P^{(u)}|$, then~$Q^{(u)}$
intersects~$P_u^{(w)}$, otherwise $Q^{(u)} \cup Q^{(u,v)} \cup P_u^{(v,w)} \cup P_u^{(w)}$ is a path of length strictly greater than~$L$ by the choice
of~$P$.
So~$P_w^{(u)}$ and $Q$ must intersect in~$C^\circ$, otherwise we derive a contradiction from Corollary~\ref{cor:3paths3} for longest paths~$P_w$
and~$Q$ with \(z=u\), tail~$P_w^{(u)}$, subpaths \(Q[C]\), and connecting path~\(P_u^{(w)}\).
\end{proof2}


\begin{proof2}[Lemma~\ref{lem:edge}]
If \(G\) is trivial, there exists a Gallai vertex. Otherwise, let \(\D= \{u,v,w\}\) be a virtual Gallai triangle, which exists by Lemma~\ref{lem:triang}.


First, we show that at least one of the edges of \(\D\) is a virtual Gallai edge.
Assume for a contradiction that no edge of \(\D\) is a virtual Gallai edge, which means that there are three longest paths \(P_u \in \LBB{u}{v}{w}\), 
\(P_v \in \LBB{v}{u}{w}\), and \(P_w \in \LBB{w}{u}{v}\). Thus, there exist three distinct components \(C_{uv}\), \(C_{uw}\), and~\(C_{vw}\) generated 
by the virtual edges of \(\D\) such that, for every $x$, $y$ in $\D$, all intersection points of \(P_x\) and \(P_y\) lie in the component \(C_{xy}\).  
Without loss of generality, let \(|P_u[C_{uv}]| \geq L / 2\). Then, by combining the paths \(P_u[C_{uv}]\), \(P_u[C_{uw}]\con{u} P_w\), and a longer 
tail of~\(P_w\), we obtain a path of length strictly greater than~$L$, a contradiction. So, there exists a virtual Gallai edge in~\(\D\).

If all edges of \(\D\) are virtual Gallai edges, then $\LBB{u}vw = \LBB{v}uw = \LBB{w}uv = \varnothing$. If moreover at least one among $\LB{u}vw$,
$\LB{u}wv$, $\LB{v}wu$ is empty, then one of the vertices in $\D$ is a Gallai vertex. Otherwise, we are in the situation of Lemma~\ref{lem:triang3}
and the statement of the lemma follows immediately.

Without loss of generality, we may assume \(\{u,v\}\) is a virtual Gallai edge. Hence, \({\LBB{w}uv = \varnothing}\). Let \(P_u \in \LBB{u}vw \cup
\LBB{v}uw\) and \(x \in \{u,v\}\) be such that~\(P_u\) has a tail starting at~\(x\) that is as long as possible. Without loss of generality, we may
assume \(x=u\) and thus \(P_u \in \LBB{u}vw\). Let~\(P_u'\) be such a longer tail. (If both tails of~\(P_u\) starting at~\(u\) have same length,
choose~\(P_u'\) to be any one of them.) Let \(P_u''\) be the other tail of~\(P_u\). As \(\LBB{u}vw\) is nonempty, \(\{v,w\}\) is not a virtual Gallai
edge. If all longest paths contain~\(u\), then we are done since \(u\) is a Gallai vertex. Otherwise, \(\LeB{v}u\) is nonempty.

Each \(P_v \in \LeB{v}u\) intersects \(P_u'\) because otherwise, by combining \(P_u'\), \(P_u'' \con{u} P_v\), and a longer tail of~\(P_v\), we get a
path of length strictly greater than~$L$ by the choice of~\(P_u\). If \(\{u,v\}\) is the only virtual Gallai edge in $\D$, then \(\LBB{v}uw \neq
\varnothing\) and \(P_u'\) has to lie in a component \(C \in \calC_{\{u,v\}|w}\), so that it intersects every \(P_v \in \LBB{v}uw\). Otherwise,
\(\{u,w\}\) is also a virtual Gallai edge and \(P'_u\) lies in a component \(C\) either in \(\calC_{\{u,v\}|w}\) or in \(\calC_{\{u,w\}|v}\). (Note
that in this case \(\LB{v}wu \neq \varnothing\).) Without loss of generality, we assume \(C \in \calC_{\{u,v\}|w}\). We claim that~\(C\) has the
desired properties.

First, we prove that each path in \(\LeB{v}u\) intersects in \(C^\circ\) every path in \(\LeB{u}v \cup \Le{u}v\).  Let \(P_v \in \LeB{v}u\). Suppose
that there exists a path \(Q \in \LeB{u}v \cup \Le{u}v\) such that \(P_v\) does not intersect \(Q\) in~\(C^\circ\). Either~\(Q\) intersects \(P_u'\)
in \(C^\circ\), or \(P_u'\) and both tails of \(Q\) starting at~\(u\) have length \(L/2\) (see Figure~\ref{fig:P2Q1}).  In the former case, since
\(P_u'\) intersects both \(P_v\) and \(Q\) in \(C^\circ\), we can apply Lemma~\ref{lem:3paths2} if \(Q \in \LeB{u}v\) (with paths~\(P_v\) and~\(Q\)
and connecting path~\(P_u'\)) or Corollary~\ref{cor:3paths3} if \(Q \in \Le{u}v\) (with paths \(P_v\) and \(Q\), \(z=v\), a suitable tail of \(P_v\)
starting at \(v\), subpaths \(Q[C]\), and a connecting path contained in~\(P_u'\)) deriving a contradiction.  So, suppose now that~\(P_u'\) and both
tails of~\(Q\) starting at~\(u\) have equal length. The path \(P_v\) intersects both tails of~\(Q\) starting at~\(u\) because otherwise such a tail of
\(Q\), \(P_u'\con{u}P_v\), and a longer tail of \(P_v\) would be a path of length strictly greater than~$L$. Therefore and since the combination of
\(P_u'\) with one tail of~\(Q\) starting at~\(u\) and the combination of $P_u'$ with the other tail of $Q$ starting at $u$ are both longest paths, we
can apply Lemma~\ref{lem:3paths} with a connecting path contained in~\(P_v\) and \(z=u\), deriving again a contradiction.  Hence, each path in
\(\LeB{v}u\) intersects in \(C^\circ\) every path in \(\LeB{u}v \cup \Le{u}v\).

\begin{figure}[!htbp]
	\centering
		\includegraphics[scale=\scalefactor,page=1]{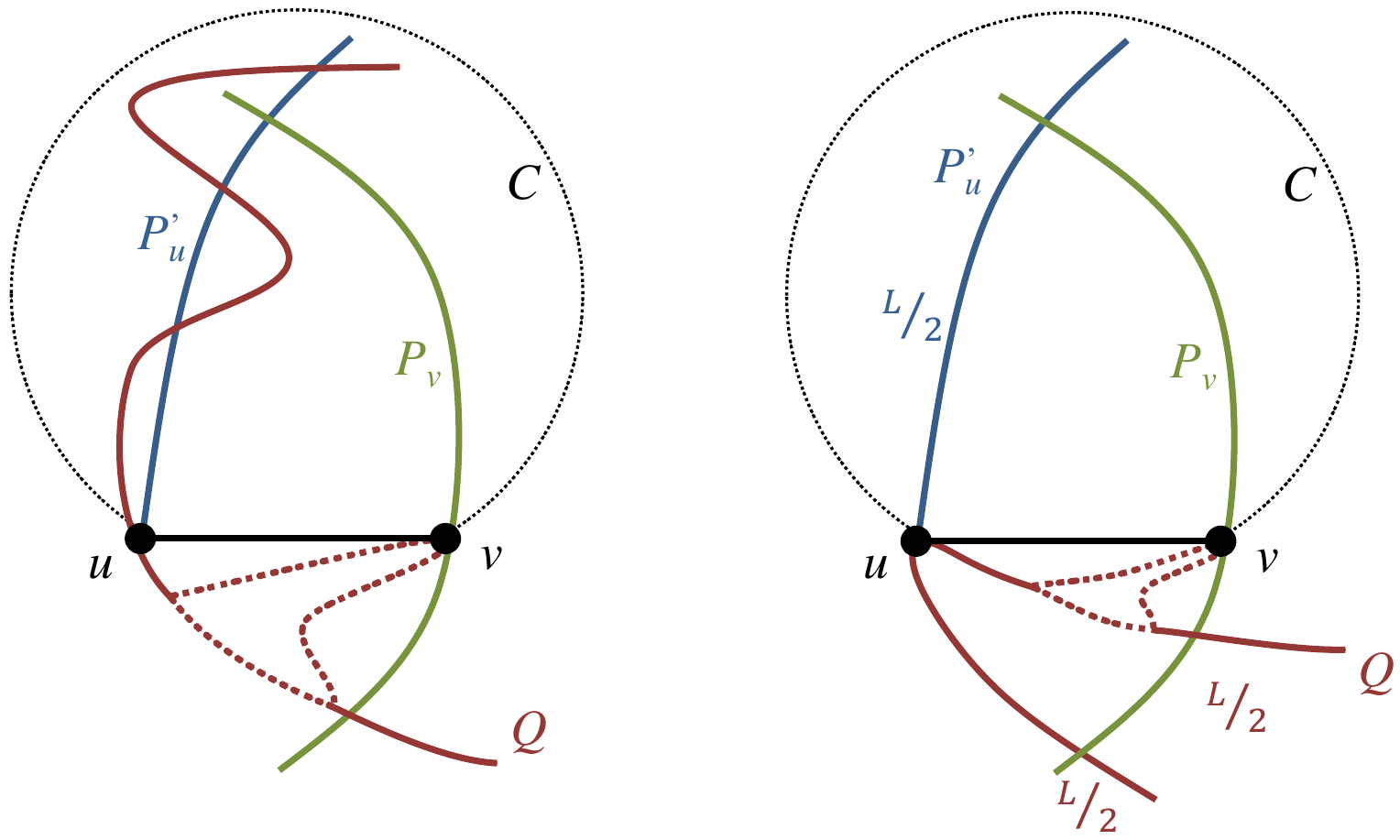}
	\caption{Left: \(P_v\) and \(Q\) do not intersect in \(C\) (apply either Lemma~\ref{lem:3paths2} or Corollary~\ref{cor:3paths3}); \\ 
                 Right: \(P_v\) intersects both tails of \(Q\) outside of \(C\) (apply Lemma~\ref{lem:3paths}).}
	\label{fig:P2Q1}
\end{figure}

Next, we prove that each path in \(\Le{u}v\) intersects in \(C^\circ\) every path in \(\LeB{u}v\).  If \(\Le{u}v \neq \varnothing\), let \(P\) be a path
in~\(\Le{u}v\) and \(Q_u\) in \(\LeB{u}v\). Assume that \(P\) does not intersect \(Q_u\) in~\(C^\circ\). 
Let \(P_v \in \LeB{v}u\) and note that such a longest path must exist. Since \(P_v\) intersects \(Q_u\) in \(C^\circ\) and \(P\) in
\(C^\circ\), we derive a contradiction longest paths \(Q_u\) and \(P\) with \(z=u\), a suitable tail of \(Q_u\) starting at $u$, subpaths~\(P[C]\),
and a connecting path contained in \(P_v[C]\).
\end{proof2}


\begin{lem}\label{lem:iter2}
  In a non-trivial connected series-parallel graph $G$, let \(e = \{u,v\}\) be a virtual Gallai edge and \({C \in \mathcal C_e}\) be a component such
    that all pairs of longest paths mutually intersect in at least one vertex of \(C\) and all pairs of longest paths in \(\LeB{u}v \times \LeB{v}u\),
    \(\LeB{u}v \times \Le{u}v\), and \(\LeB{v}u \times \Le{u}v\) mutually intersect in \(C^{\circ}\) (as in Lemma~\ref{lem:edge}). Let $w$ be the
    unique vertex in $C$ adjacent in $T(G)$ to both $u$ and $v$.  Then \(u\), \(v\), or \(w\) is a Gallai vertex, or there is a virtual edge \(f\)
    incident to \(u\) or \(v\) and a component \(C_1 \in \mathcal C_f\), \(C_1 \subsetneq C\), with the properties of Lemma~\ref{lem:edge}.
\end{lem}

\begin{proof}
Let $\Delta = \{u,v,w\}$.  Assume neither $u$ nor $v$ are Gallai vertices (otherwise there is nothing more to prove).  Thus, both \(\LeB{u}v\) and
\(\LeB{v}u\) are nonempty.  All pairs of paths in \(\LeB{u}v \times \LeB{v}u\) intersect in \(C^\circ\) by the assumptions of the lemma. By
Lemma~\ref{lem:SP}, all paths in \(\LeB{u}v\) or all paths in \(\LeB{v}u\) must contain \(w\) since $C \notin \calC_{\{u,v\}|w}$. Without loss of
generality we may assume that all paths in $\LeB{u}v$ contain $w$. Therefore, \(\LBB{u}vw = \varnothing\) and \(\{v,w\}\) is a virtual Gallai edge.

We distinguish two cases. First, we consider the case in which there exists a path in \(\LeB{v}u\) that does not contain \(w\) (that is, a path in
$\LBB{v}uw$) and then the case in which all paths in \(\LeB{v}u\) contain \(w\) (that is, $\LBB{v}uw = \varnothing$). In both cases we show that either
there exists a Gallai vertex, or a virtual Gallai edge and a component strictly smaller than~\(C\) that fulfill the requirements of
Lemma~\ref{lem:iter2}.

\begin{figure}[!htbp]
	\centering
		\includegraphics[scale=\scalefactor,page=1]{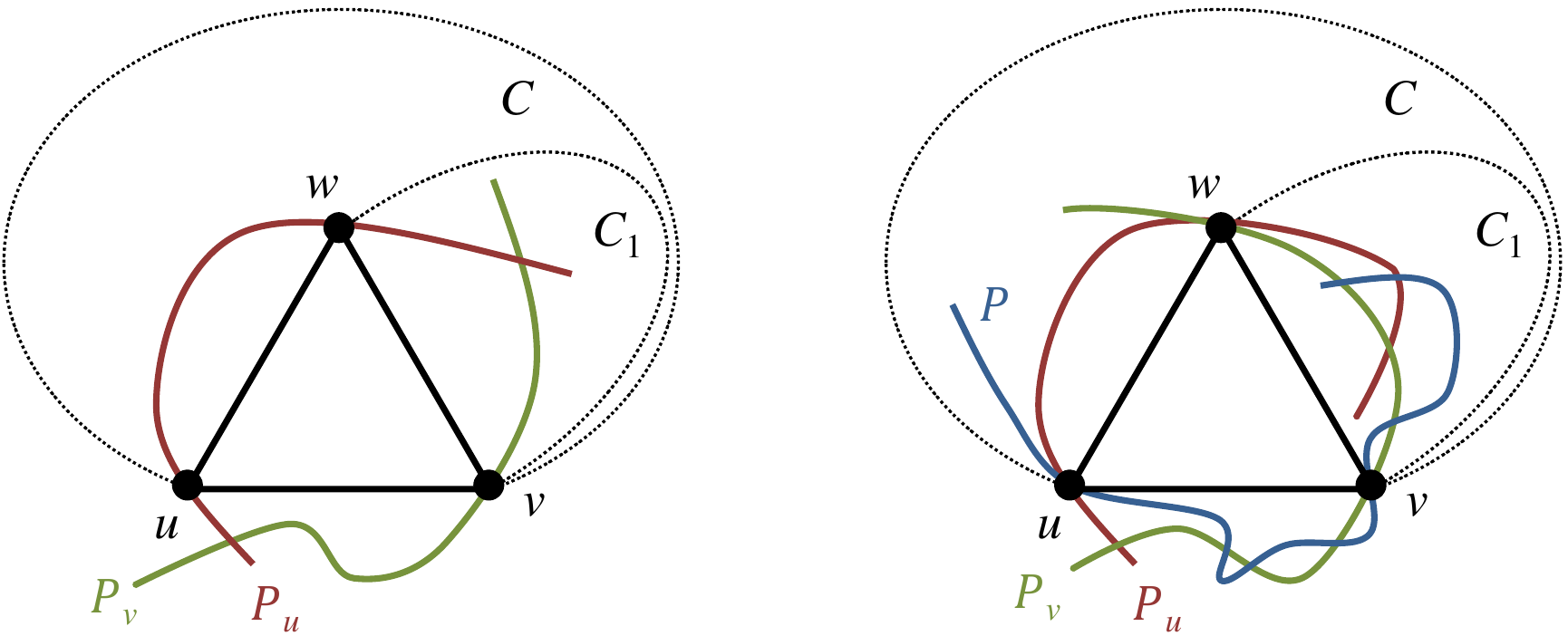}
	\caption{The scenario of the proof of Lemma~\ref{lem:iter2}: Case 1 (left) and Case 2 (right).}
	\label{fig:indLemma}
\end{figure}

\textit{Case 1. The set \(\LBB{v}uw\) is nonemtpy.}

For every path \(P_u \in \LeB{u}v = \LB{u}wv\), the tail \(P_u^{(w)}\) must intersect every path in \(\LBB{v}uw\) by assumption. Let \(C_1 \subsetneq C\),
\(C_1 \in \mathcal C_{\{v,w\}|u}\) be the unique component where they mutually intersect.

We claim that the virtual edge \(\{v,w\}\) together with the component \(C_1\) fulfills the requirements of Lemma~\ref{lem:iter2}.

If \(\LB{u}vw \neq \varnothing\), let \(P \in \LB{u}vw\) and \(P_v\in \LBB{v}uw\).  Assume for a contradiction that there exists a path \(P_u \in
\LeB{u}v\) such that \(P\) does not intersect \(P_u\) in \(C_1^\circ\). Note that \(P_v\) and \(P\) must intersect in \(C^\circ\) by assumption and
hence \(P_u\) and \(P^{(v)}\) are disjoint. Since \(P_v\) intersects \(P_u\) in \(C_1^\circ\) and \(P\) in \(C_u\) (at least in vertex \(v\)), we can
apply Lemma~\ref{lem:3paths2} (with paths \(P_u\) and \(P\), tails \(P_u^{(w)}\) and \(P^{(v)}\), and a connecting path contained in \(P_v[C_1]\)) to
derive a contradiction.

Next, we prove that every path in \(\LB{v}wu \cup \Luvw{u}vw\) intersects every path in \(\LBB{v}uw \cup \LeB{u}v \cup \LB{u}vw\) in~\(C_1^\circ\).

Let \(P \in \LB{v}wu \cup \Luvw{u}vw\) and \(P_v \in \LBB{v}uw\). Note that \(P_v\) intersects \(P\) in \(C^\circ\) by assumption if \(P \in
\Luvw{u}vw\). Otherwise, \(P \in \LB{v}wu\), and they also intersect in \(C^\circ\), or not both \(P_v\) and \(P\) could be longest paths by
Lemma~\ref{lem:3paths} for \(z=v\), a suitable tail of \(P_v[C_1]\), tail \(P^{(v,w)} \cup P^{(w)}\), and connecting path~\(P_u^{(w)}\) for some \(P_u
\in \LeB{u}v\).  Furthermore, \(P\) must intersect \(P_v\) in \(C_1^\circ\). Otherwise, \(P_v\) would have a tail starting at \(v\) completely in a
component \(C' \in \mathcal C_{\{v,w\}|u}\), \(C' \neq C_1\).  Then \(|P_v[C']| < L/2\) since~\(P_v[C']\) is disjoint from \(P_u\) and so \(P_v[C']
\cup P_v[C_1] \con{v} P_u\), and a longer tail of \(P_u\) would be a path of length strictly greater than~$L$. But now, by combining \(P_v[C_1]\) and
a longer tail of \(P\) starting at \(v\), we get a path of length strictly greater than~$L$, a contradiction. For every \(P_u \in \LeB{u}v\), by
applying Corollary~\ref{cor:3paths3} (with paths~\(P_u\) and~\(P\), \(z=w\), tail~\(P_u^{(w)}\), subpaths~\(P[C_1]\), and connecting
path~\(P_v[C_1]\)), we can deduce that \(P\) intersects~\(P_u\) in~\(C_1^\circ\). Every \(P_w \in \LB{u}vw\) must intersect~\(P\) in~\(C_1^\circ\),
otherwise we get a contradiction by applying Corollary~\ref{cor:3paths3} (with paths \(P\) and~\(P_w\), \(z=v\), tail~\(P_w[C_1]\),
subpaths~\(P[C_1]\), and connecting path~\(P_u^{(w)}\)).

\textit{Case 2.  The set \(\LBB{v}uw\) is empty.}

If the set \(\LB{u}vw\) is empty, then all longest paths contain \(w\), therefore \(w\) is a Gallai vertex and the requirements of
Lemma~\ref{lem:iter2} are fulfilled. So, from now on, we assume that the set \(\LB{u}vw\) is nonempty.

First, we prove that, for each \(P \in \LB{u}vw\), either \(P^{(v)}\) intersects \(P_u^{(w)}\) for every \(P_u \in \LeB{u}v = \LB{u}wv\), or
\(P^{(u)}\) intersects \(P_v^{(w)}\) for every \(P_v \in \LeB{v}u = \LB{v}wu\). Assume for a contradiction that there exist \(P \in \LB{u}vw\), \(P_u
\in \LeB{u}v\), and \(P_v \in \LeB{v}u\) such that \(P^{(u)}\) does not intersect \(P_v^{(w)}\) and \(P^{(v)}\) does not intersect~\(P_u^{(w)}\). By
the assumptions of the lemma, \(P\) has to intersect both~\(P_u\) and \(P_v\) in \(C^\circ\). Therefore, \(P\) intersects \(P_u\) in the interior of
some component of~\(\mathcal C_{\{u,w\}|v}\) and~\(P_v\) in the interior of some component of~\(\mathcal C_{\{v,w\}|u}\). (See Figure~\ref{fig:indLemma2}.)  
First, suppose that \(P_u^{(u)}\) and \(P_v^{(v)}\) do not intersect. By combining \(P^{(v)}\), \(P^{(v,u)} \con{v} P_u\), and a longer tail of
\(P_u\), we get a path that cannot be of length strictly greater than~$L$, hence \(|P^{(v)}| < L/2\). Combining \(P^{(u)} \cup P^{(u,v)} \cup
P_v^{(v,w)} \cup P_v^{(w)}\) would be a path of length strictly greater than~$L$ if \(|P_v^{(v)}| \leq L/2\). However, by combining \(P_v^{(v)}\),
\(P_v^{(v,w)} \con{v} P_u\), and a longer tail of \(P_u\), we get a path of length strictly greater than~$L$, a contradiction. If, on the other
hand, \(P_u^{(u)}\) and \(P_v^{(v)}\) intersect, then \(P^{(u)}\) and \(P_u^{(u)}\) are disjoint except for the vertex \(u\). Note that \(|P^{(u)}|
= |P_u^{(u,w)} \cup P_u^{(w)}|\) since otherwise $P$ and $P_u$ would not be longest paths. The combination of \(P_u^{(u)}\) and \(P^{(u)}\) is
therefore a longest path in \(\Luvw{u}vw\), which is a contradiction. Therefore, \(P^{(u)}\) intersects~\(P_v^{(w)}\), or~\(P^{(v)}\) intersects
\(P_u^{(w)}\).

\begin{figure}[!htbp]
	\centering
		\includegraphics[scale=\scalefactor,page=1]{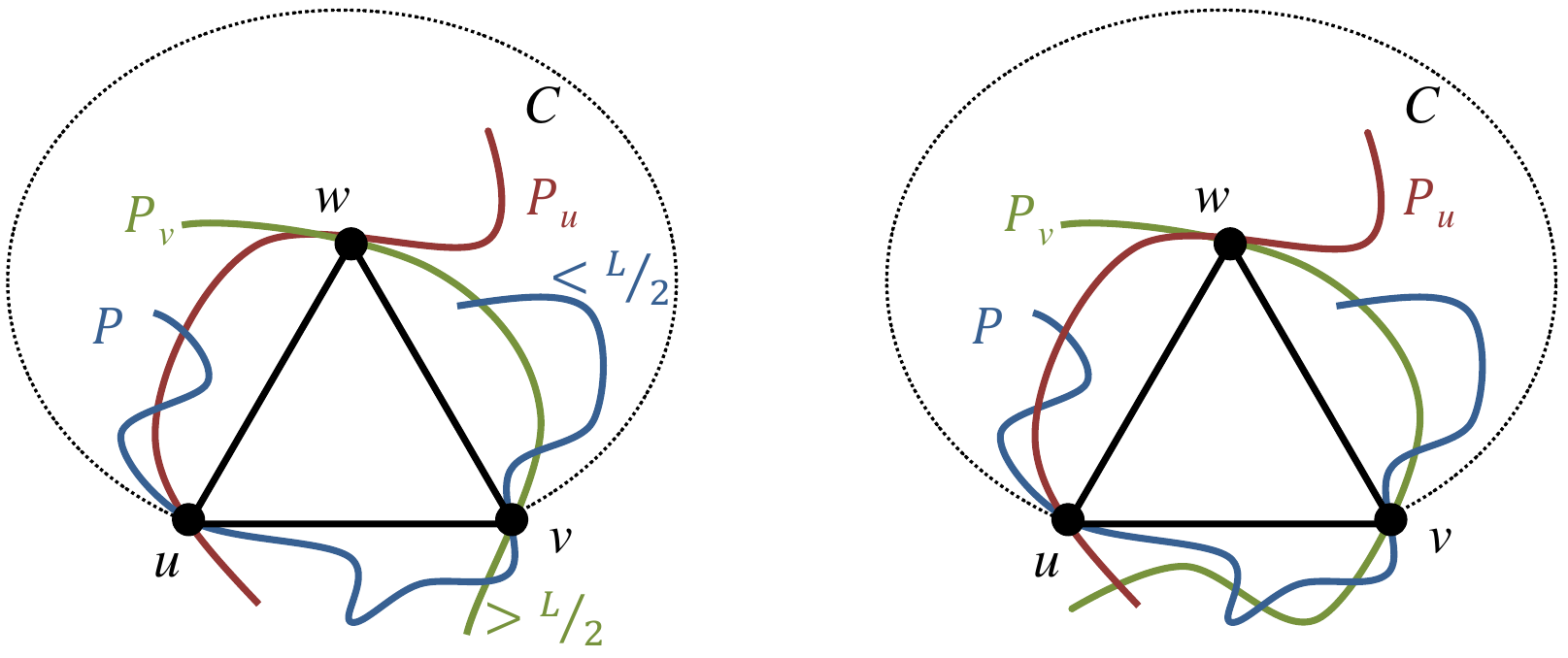}
	\caption{Left: \(P_u^{(u)}\) and \(P_v^{(v)}\) do not intersect; Right: \(P_u^{(u)}\) and \(P_v^{(v)}\) intersect.}
	\label{fig:indLemma2}
\end{figure}

We claim that if \(R^{(u)}\), for some longest path \(R\) in \(\LB{u}vw\), does not intersect \(P_v^{(w)}\) for some path \(P_v \in \LeB{v}u\), then
for every longest path \(P \in \LB{u}vw\), the tail \(P^{(v)}\) intersects \(P_u^{(w)}\) for every path \({P_u \in \LeB{u}v}\). Indeed, let \(P\) be a path
in~\(\LB{u}vw\). Assume for a contradiction that there exists a path \(P_u \in \LeB{u}v\) such that \(P^{(v)}\) does not intersect~\(P_u^{(w)}\). Note
that by the latter paragraph \(P_u^{(w)}\) must intersect \(R^{(v)}\) and~\(P_v^{(w)}\) must intersect~\(P^{(u)}\). If \(P^{(v)}\) lies in some
component of \(\mathcal C_{\{v,w\}|u}\), then by Lemma~\ref{lem:SP} $P^{(v)}$ cannot intersect $R^{(u)}$ and $R^{(u,v)}-v$, and $R^{(v)}$ cannot
intersect $P^{(u)}$ and $P^{(u,v)}-v$ since $R^{(v)}$ also lies in a component of \(\mathcal C_{\{v,w\}|u}\). Hence, we have \(|P^{(v)}| = |R^{(v)}|\)
and \(Q = R^{(u)} \cup R^{(u,v)}\cup P^{(v)}\) is a path in \(\LB{u}vw\) such that \(Q^{(u)}\) does not intersect \(P_v^{(w)}\), and \(Q^{(v)}\) does
not intersect~\(P_u^{(w)}\), a contradiction. Analogously, if \(R^{(u)}\) lies in a component of \(\mathcal C_{\{u,w\}|v}\), the path \(R^{(u)} \cup
P^{(u,v)}\cup P^{(v)}\) yields a contradiction. Thus, we may assume that both \(R^{(u)}\) and \(P^{(u)}\) lie in a component of~\(C_{\{u,v\}|w}\) and
are therefore disjoint from \(P^{(u)}\) except for \(u\), and from \(R^{(v)}\) except for~\(v\), respectively. Now, we get a contradiction from
Lemma~\ref{lem:3paths2} for longest paths \(R\) and \(P\), tails~\(R^{(v)}\) and~\(P^{(u)}\), and a connecting path contained in \((P_u^{(w)}\con{w}R)
\cup (P_v^{(w)}\con{w}P)\).

Without loss of generality, we may assume that, for each \(P \in \LB{u}vw\), the tail \(P^{(v)}\) intersects \(P_u^{(w)}\) for every \(P_u \in \LeB{u}v\).
Let \(C_1 \in \mathcal C_{\{v,w\}|u}\) be the unique component where they mutually intersect. Note that the virtual edge \(\{v,w\}\) is indeed a
virtual Gallai edge since the set \(\LBB{u}vw\) is empty. Let \(P \in \LB{u}vw\) and \(P_u \in \LeB{u}v\) be arbitrary but fixed. Note that \(P_u\)
intersects \(P\) in \(C_1^\circ\). Assume for a contradiction that there exists a longest path \(P_v \in \LB{v}wu \cup \Luvw{u}vw\) such that \(P\)
and \(P_v\) do not intersect each other in \(C_1^{\circ}\).  Then not both \(P\) and \(P_v\) can be longest paths by Corollary~\ref{cor:3paths3} 
for~\(z=v\), tail \(P^{(v)}\), subpaths \(P_v[C_1]\), and connecting path \(P_u^{(w)}\), a contradiction. Therefore, the virtual edge \(\{v,w\}\) 
together with the component \(C_1\) fulfills the requirements of Lemma~\ref{lem:iter2}.
\end{proof}


\begin{thm}
For every connected series-parallel graph \(G\), there exists a vertex \(v\) such that all longest paths in \(G\) contain~\(v\). 
\end{thm}

\begin{proof}
This follows from Lemma~\ref{lem:edge} and by iteratively applying Lemma~\ref{lem:iter2} since \(G\) is finite. 
\end{proof}

\section{Algorithmic remarks}
\label{sec:algo}

For any hereditary class of graphs for which there is a polynomial-time algorithm that computes (the length of) a longest path, it is easy
to derive a polynomial-time algorithm that finds all Gallai vertices. Indeed, one just has to compute the length $L$ of a longest path in
the given (connected) graph~$G$, and then to check, for each vertex $v$, whether the length of a longest path in $G-v$ remains the same. If
not, $v$ is a Gallai vertex.

It is a well-known result that one can use dynamic programming to solve many combinatorial problems on graphs of bounded treewidth in
polynomial or even linear time~\cite{arnborg,bodlaender}. In particular, Bodlaender~\cite[Thm. 2.2]{Bodlaender93} claims a linear-time
algorithm following these lines to find a longest path in a graph with bounded treewidth. (See also~\cite{Bodlaender96} on how to obtain in
linear time a tree decomposition for graphs with bounded treewidth.) Therefore, using the idea described in the previous paragraph, one can
find all Gallai vertices in time quadratic on the number of vertices of the given connected series-parallel graph.

In fact, one can do better by applying the same strategy used to compute the length of a longest path in a partial $k$-tree, but carrying
more information during the process. Given a connected graph $G$ of treewidth $k$, compute in linear time a (nice) tree decomposition
for~$G$ (as done in~\cite{DiazST02} for instance). Then run a dynamic programming algorithm on top of this tree decomposition, to compute
the length $L$ of a longest path in $G$. Roughly speaking, this algorithm computes the length of longest parts of paths within the subgraph
induced by the vertices in clusters already traversed of the tree decomposition, and puts together this information while going through the
tree decomposition. Specifically, when visiting a node $u$ of the tree decomposition, if $H_u$ is the subgraph of $G$ induced by the
vertices in the cluster $X_u$ or clusters of nodes below $u$, for each different way that a path can behave in the cluster~$X_u$ and
in~$H_u$, we have a configuration as the ones described in Figure~\ref{fig:configurations} for the case in which $X_u$ has three vertices.
\begin{figure}[htbp]
  \centering
  \includegraphics[scale=\scalefactor]{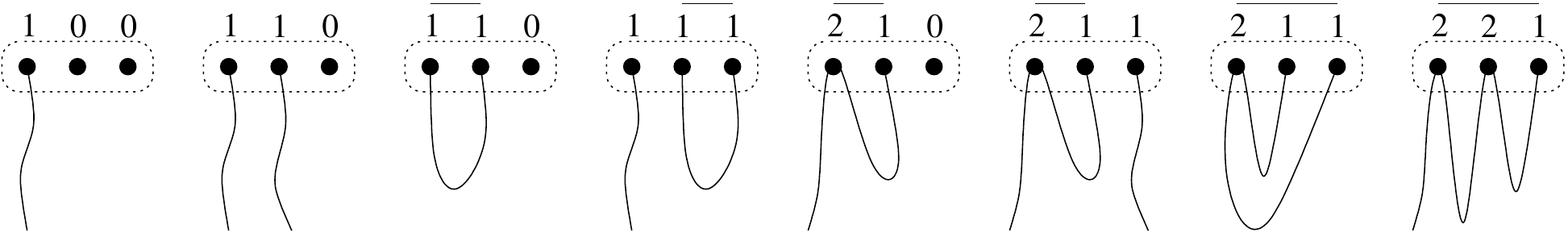}
  \caption{A sample of the configurations for a cluster with three vertices. The whole set of configurations has to consider the labels of the vertices in the cluster.}
  \label{fig:configurations}
\end{figure}

The number of such configurations depends only on the treewidth. For each such configuration, the algorithm computes the length of a longest
part of a path in~$H_u$ that ``agrees'' with that configuration. It does this using dynamic programming, that is, it computes such length
for a node $u$ and one of the configurations using the information that it already computed for the children of $u$ in the tree. Some of the
configurations of the children, together with new edges within $X_u$, combine into each configuration for $u$. The combinations that give
raise to the longest parts are the ones of interest, and give the length of a longest part for that configuration for $u$.

In a first traversal of the tree decomposition, the value of~$L$ is computed. Now, as it is usual in dynamic programming, in a reverse
traversal of the tree, retracing backwards what was done to find out~$L$, one can mark, for each node and each configuration, if that
configuration at that node gives raise to a path of length~$L$ in~$G$. Once this is done for a node, the algorithm checks whether the
configurations for that node that give raise to a longest path all contain one of the vertices in the cluster of that node. If so, this is a
Gallai vertex. Otherwise the algorithm proceeds to the next node in the reverse traversal.

This process finishes with a Gallai vertex as long as the graph has one such vertex. In particular, for partial 2-trees, this process will
find a Gallai vertex in the reverse traversal as soon as it reaches the first cluster of the tree that contains a Gallai vertex. By
proceeding with the reverse traversal in this way, one can find all Gallai vertices. For bounded $k$, the running time of this algorithm is
linear on the number of vertices of the graph. (Note that the number of edges in a partial $k$-tree is at most $kn$, where $n$ is the number
of vertices in the partial $k$-tree.) Indeed, first computing a (nice) tree decomposition can be done in linear time. Second, the number of
configurations depends only on $k$, and the processing of each node of the tree decomposition depends only on the number of configurations
(and on the size of the cluster, which is bounded by $k+1$ and thus also by the number of configurations). Therefore, for series-parallel
graphs, this algorithm finds a Gallai vertex in time that is linear on the number of vertices of the graph.

\section{Related results and open questions}
\label{sec:question}

There are several questions related to Gallai's original question that remain open. 
For instance, it was asked~\cite{Kensell11,Zamfirescu01} whether there is a vertex common to all longest paths in all 4-connected graphs. 
This problem is open so far, and even the more general question for \(k\)-connected graphs with larger $k$ has not been answered. There are 3-connected examples known for which Gallai's question has a negative answer~\cite{Grunbaum74}.

In~\cite{deRezendeFMW13}, where a proof that all 2-trees have nonempty intersection of all longest paths was presented, it was asked whether the same holds for $k$-trees with larger values of~\(k\). As far as we know, this also has not yet been answered. In the present paper, we have proven that all connected subgraphs of 2-trees have nonempty intersection of all longest paths. We observe that the same does not hold for all subgraphs of 3-trees. Indeed, the counterexample by Walther, Voss, and Zamfirescu in~\cite{Walther74,Zamfirescu76} is a connected spanning subgraph of a 3-tree (see Figure~\ref{fig:wzamf2}).

\begin{figure}[htbp]
  \centering
  \includegraphics[scale=\scalefactor]{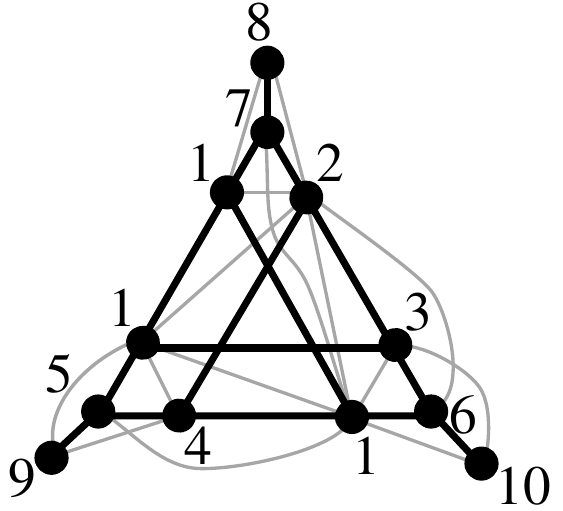}
  \caption{The counterexample of Walther, Voss, and Zamfirescu as a subgraph of a 3-tree. Missing edges are dotted. The number next to each vertex indicates the sequence in which they are added to the 3-tree.}
  \label{fig:wzamf2}
\end{figure}

In other words, Gallai's question has a positive answer for connected graphs with treewidth at most~$2$ (series-parallel graphs), but a negative answer for connected graphs with treewidth at most~$3$. 
As series-parallel graphs are the class of $K_4$ minor free graphs, one might also ask whether the answer is positive for all (connected) $K_5$ minor free graphs, but there are planar counterexamples known~\cite{Thomassen76}. 

As split graphs and 2-trees are chordal, a natural question raised by Balister et al.~\cite{BalisterGLS04} is whether all longest paths share a vertex in all chordal graphs. Recently, Michel Habib (personal communication) suggested that the answer to Gallai's question might be positive in co-comparability graphs. For this class of graphs, as well as for series-parallel graphs, there is a polynomial-time algorithm to compute a longest path~\cite{IoannidouN13}. (For chordal graphs, computing a longest path is NP-hard~\cite{Muller96}.) 

As already stated in Section~\ref{sec:intro}, instead of looking at the intersection of all longest paths, Zamfirescu asked whether any~\(p\) longest paths in an arbitrary connected graph contain a common vertex. This is certainly true for \(p=2\), proven to be false \cite{Schmitz75,Skupien96} for \(p\geq7\), but still open for \(p\) in \(\{3,4,5,6\}\).

\bibliographystyle{amsplain} 	 
\bibliography{longest}	 		 	 	 
\clearpage

\end{document}